\documentclass[12pt, reqno]{amsart}
\usepackage{amsmath} \usepackage{dsfont} \usepackage{amsfonts} \usepackage{amssymb} \usepackage{aliascnt} \usepackage{comment} \usepackage{mathrsfs} \usepackage{enumerate} \usepackage[bookmarks=true,pdfstartview=FitH, pdfborder={0 0 0}, colorlinks=true,citecolor=blue, linkcolor=blue]{hyperref}

\newtheorem{theorem}{Theorem}[section]
\newaliascnt{conj}{theorem}
\newaliascnt{cor}{theorem}
\newaliascnt{lemma}{theorem}
\newaliascnt{fact}{theorem}
\newaliascnt{claim}{theorem}
\newaliascnt{prop}{theorem}
\newaliascnt{definition}{theorem}

\newtheorem{prop}[prop]{Proposition}

\newtheorem{definition}[definition]{Definition}

\aliascntresetthe{conj} \aliascntresetthe{cor}
\aliascntresetthe{lemma} \aliascntresetthe{fact}
\aliascntresetthe{claim} \aliascntresetthe{prop}
\aliascntresetthe{definition}

\theoremstyle{definition}
\newaliascnt{example}{theorem}

\aliascntresetthe{example}

\theoremstyle{remark}
\newaliascnt{rmk}{theorem}
\newtheorem{remark}[rmk]{Remark}
\aliascntresetthe{rmk} 

\def\sek~{\S{}}

\numberwithin{equation}{section}
\newcommand{\bisec}{{\rm bisec}}
\newcommand{\kahler}{K\"{a}hler}
\newcommand{\sk}{\textbf{\textsf{s}}}
\newcommand{\ck}{\textbf{\textsf{c}}}
\newcommand{\KH}{\textsf{KH}_{k}}
\renewcommand{\H}{\textsf{H}_{k}}

\newcommand{\hess}{\operatorname{Hess}}

\newcommand{\dis}{\operatorname{d}}
\newcommand{\dist}{\operatorname{d}}
\newcommand{\Ric}{\operatorname{Ric}}

\newcommand{\II}{\operatorname{II}}
\newcommand{\RR}{\mathds{R}}

\renewcommand{\SS}{\mathbf{S}}

\newcommand{\CP}{\mathds{CP}}

\begin{document}
\title[Comparison for Manifold with Boundary]{Comparison Theorems for Manifold with Mean Convex Boundary}
\author{Jian Ge}
\maketitle
\begin{abstract}
Let $M^n$ be an $n$-dimensional Riemannian manifold with boundary $\partial M$. Assume that Ricci curvature is bounded from below by $(n-1)k$, for $k\in \RR$, we give a sharp estimate of the upper bound of $\rho(x)=\dis(x, \partial M)$, in terms of the mean curvature bound of the boundary. When $\partial M$ is compact, the upper bound is achieved if and only if $M$ is isometric to a disk in space form. A \kahler\ version of estimation is also proved. Moreover we prove a Laplace comparison theorem for distance function to the boundary of \kahler\ manifold and also estimate the first eigenvalue of the real Laplacian.
\end{abstract}
\section{Introduction}
Let $M^n$ be an $n$-dimensional Riemannian manifold with boundary $\partial M$, $\rho(x)=\dist(x, \partial M)$ be the distance function to the boundary. In a recent paper \cite{Li2012}, the author proved $\max_{x\in M}\rho(x)\le 1/k$ under the assumptions that the Ricci curvature $\Ric \ge 0$ and the mean curvature $H$ of the boundary satisfies $H \ge (n-1)k$ for $k>0$, the equality holds if and only if $M$ isometric to the Euclidean ball of radius $1/k$. The argument is essentially the well known Jacobi field estimates. The idea can be traced back to \cite{HK1978}, although estimations of volume is main topic there instead of the distance functions, also they treated compact manifold only. Under a stronger assumption that the sectional curvatures bounded from below, Dekster estimated $\rho$ and lenght of more general curves in \cite{Dek1977}. Alexander and Bishop studied more general Alexandrov spaces with certain convexity condition on the boundary in \cite{AB2010}, among many other things, one estimate of the upper bound of $\rho$ has also been derived for such spaces. In this note we generalize the above mentioned theorem to Riemannian manifold with lower Ricci curvature bound $k\in \RR$ and give an unified proof for all $k$. i.e. we prove the following
\begin{theorem}\label{thm:k}
Let $M^n$ be a complete $n$-dimensional Riemannian manifold with lower Ricci curvature bound $(n-1)k$ and boundary $\partial M$. Assume the mean curvature $H$ of the boundary satisfies $H\ge (n-1)h$. For the case $k\le 0$ we assumer further that $h>\sqrt{-k}$. Let $\rho(x)=\dis(x, \partial M)$ be the distance to the boundary. Then

\begin{equation} \label{eq:2}
\rho(x)\le \left\{ \begin{aligned}
\frac{1}{\sqrt{|k|}}&\coth^{-1}\Big(\frac{h}{\sqrt{|k|}}\Big) \quad&{\rm for }\quad k<0 \\
&\frac{1}{h}\quad &{\rm for } \quad  k=0\\
\frac{1}{\sqrt{k}}&\cot^{-1}\Big(\frac{h}{\sqrt{k}}\Big)\quad &{\rm for } \quad k>0\\
\end{aligned} \right.
\end{equation}
\end{theorem}

If we assume further that $\partial M$ is bounded, then the upper bound in \eqref{eq:2} implies that $M$ is compact. Note also that the condition $h> \sqrt{-k}$ for $k\le 0$ is also sharp. Since the distance function $\rho(x)$ defined on the warp product $[0, \infty) \times_{e^{-t}} \SS^{n-1}$ is unbounded and satisfies the mean convex condition for $h=\sqrt{-k}=1$. Note one can also give a Laplace comparison of $\rho$ for manifold with lower Ricci curvature bound and mean convex boundary, which is implicitly proved in \cite{HK1978}. Hence we will only list theorem here and omitted the proof. However, the proof is similar as the \kahler\ version comparison proved in section 2. Define
$$
\H(r)=(n-1)\ck_{k}(r)/\sk_{k}(r),
$$
which is the mean curvature of the geodesic sphere of radius $r$ in simply connected space form with constant sectional curvature $k$, for $r\in (0, \pi/\sqrt k)$ if $k > 0$.  See \eqref{eq:sk} and \eqref{eq:ck} for the definition of $\sk_k$ and $\ck_k$. Hence we have
\begin{theorem}[Implicitly given in \cite{HK1978}]\label{thm:realLapComp}
Let $M^n$ be a complete $n$-dimensional Riemannian manifold with lower Ricci curvature bound $(n-1)k$ and boundary $\partial M$. Assume the mean curvature $H$ of the boundary satisfies $H\ge \H(h)>0$. Let $\rho(x)=\dis(x, \partial M)$ be the distance to the boundary. Then
$$
\Delta\rho(x)\le -\H(h - \rho(x)).
$$
\end{theorem}

Let $\SS^n_k$ be the simply connected $n$-dimensional space form of constant sectional curvature $k$. Let $D_X(p, r)$ denote the close disk of radius $r$ centered at point $p\in X$. i.e. $ D_X(p, r)=\{x\in X| \dist(x, p)\le r\}$. The proof of the rigidity part in \cite{Li2012} for $\Ric\ge 0$ does not extend to general lower curvature bound, hence we use \autoref{thm:realLapComp} and the idea in Cheeger-Gromoll's proof of splitting theorem \cite{CG1971}, one can prove the following rigidity theorem for \autoref{thm:k}.
\begin{theorem}\label{thm:rigidity}
Let $M^n$ be a complete $n$-dimensional Riemannian manifold with lower Ricci curvature bound $(n-1)k$ and boundary $\partial M$. Assume the mean curvature $H$ of the boundary satisfies $H\ge \H(h)$. Let $\rho(x)=\dis(x, \partial M)$ be the distance to the boundary. Let $\ell=\max_{x\in M}\rho(x)$, then $\ell=h$ if and only if $M$ is isometric to $D_{\SS^n_k}(h)$.
\end{theorem}

It is well known that results related to Riemannian curvatures are sometimes also hold for \kahler\ manifold when hypothesis are suitably phrased in terms of bi-sectional or holomorphic sectional curvature. In fact the idea of the proof of \autoref{thm:k} can be used to prove the following estimate for \kahler\ manifold. It suits our purposes well in this note to avoid complex vector spaces. In fact we will treat \kahler\ manifold as a Riemannian manifold with metric $g$ admitting a parallel skew-symmetric linear transformation $J$ on the tangent bundle such that $J^2=-I$, where $I$ is the identity transformation of $TM$. Let $R$ be the Riemannian curvature tensor of $g$ with the convention
$$
R(X, Y, Z, W)=\langle -\nabla_X\nabla_Y Z +\nabla_Y\nabla_X Z +\nabla_{[X, Y]}Z, W\rangle.
$$
Hence the sectional curvature of a plane $\sigma$ is $R(X, Y, X, Y)$, where $X, Y$ are an orthonormal basis of $\sigma$. Let $\sigma_1$ and $\sigma_2$ are two planes in $T_xM$ each invariant under $J$. Let $X$ and $Y$ be unit vectors in $\sigma_1$ and $\sigma_2$ respectively. Then the holomorphic bisectional curvature of $\sigma_1$ and $\sigma_2$ is defined by:
$$
\bisec(\sigma_1, \sigma_2)=R(X, JX, Y, JY).
$$
By the Bianchi identity we have
$$
\bisec(\sigma_1, \sigma_2)=R(X, Y, X, Y)+R(X, JY, X, JY)
$$
Notice that our definition of bi-sectional curvature differs by factor $2$ to the one gave in \cite{LW2005}. By writing the inequality $\bisec \ge 2k$ we mean
$$
R(X, Y, X, Y)+R(X, JY, X, JY)\ge 2k(|X|^2|Y|^2+\langle X, Y\rangle^2 +\langle X, JY\rangle ^2),
$$
for any $X, Y\in T_xM$ and any $x\in M$. Note that the complex projective space $\CP^n$ with the Fubini-study metric is normalized to have constant holomorphic sectional curvature $4$. The condition $\bisec \ge 2k$ are stronger than the condition $\Ric \ge (2n+2)k$, however the conclusions one can get are also stronger. Define
$$
\KH(r)=\ck_{4k}(r)/\sk_{4k}(r)+(2n-2)\ck_{k}(r)/\sk_{k}(r),
$$
which is the mean curvature of the geodesic sphere of radius $r$ in simply connected complex space form with constant holomorphic sectional curvature $4k$. One easily verify that $\KH(r)$ is monotone decreasing in it's natural domain of definition. We have the following \kahler\ version of \autoref{thm:k}
\begin{theorem}\label{thm:kahler}
Let $M^n$ be a \kahler\ manifold of complex dimension $n$ with bisectional curvature $\bisec \ge 2k$, for $k\in \RR$. If $\partial M$ is mean convex with mean curvature $H\ge \KH(h)$, if $k=1$ we assume $0<h<\pi/2$. Let $\rho(x)=\dis(x, \partial M)$ be the distance to the boundary. Then for all $x\in M$,
$$
\rho(x) \le h.
$$
\end{theorem}

In fact if we put a stronger assumption on the convexity of $\partial M$, then \autoref{thm:kahler} can also be viewed as a consequence of the following Laplacian comparison theorem for $\rho(x)$:
\begin{theorem}\label{thm:kahlercomp}
Let $M^n$ be a \kahler\ manifold of complex dimension $n$ with bisectional curvature $\bisec \ge 2k$, for $k\in \RR$. If the second fundamental form of $\partial M$ with respect to inner normal direction $\nu$ satisfy $\II(J\nu, J\nu)\ge \ck_{4k}(h)/\sk_{4k}(h)$ and $\II(V, V) + \II(JV, JV) \ge 2 \ck_{k}(h)/\sk_{k}(h)$ for any $V\perp J\nu$, if $k=1$ we assume $0<h<\pi/2$. Let $\rho(x)=\dis(x, \partial M)$ be the distance to the boundary. Then
\begin{equation*}
\Delta\rho(x)\le - \KH(h-\rho(x)) \le -\KH(h).
\end{equation*}
\end{theorem}

The Laplace comparison for distance function to a point with arbitrary holomorphic bisectional curvature lower bound is given by Li-Wang \cite{LW2005}.
\begin{theorem}[Li-Wang's Laplace Comparison, \cite{LW2005}]
Let $M^n$ be a \kahler\ manifold of complex dimension $n$ with holomorphic bisectional curvature $\bisec \ge 2k$, for $k\in \RR$. Let $r(x)=\dist(x, o)$, where $o\in M$ is a fixed point in $M$. Then
\begin{equation*}
\Delta r(x)\le \KH(r(x)).
\end{equation*}
\end{theorem}

We also note that a complex Hessian comparison theorem for the Busemann function on complete \kahler\ manifold with non-negative holomorphic bisectional curvature was proved by Greene-Wu \cite{GW1978}. Cao-Ni \cite{CN2005} proved the complex Hessian comparison theorem for the distance function on \kahler\ manifold with non-negative holomorphic bisectional curvature. Our techniques used in the proofs of \autoref{thm:kahler} and \autoref{thm:kahlercomp} are inspired by \cite{GW1978} and \cite{CS2005}. The Laplacian comparison of distance function to the boundary is a natural complement of the comparison theorems given in \cite{LW2005}. As an application of the Laplace comparison, we estimate the lower bound of the first eigenvalue of the Laplacian for \kahler\ manifold with boundary using the idea of Song-Ying Li and Xiaodong Wang \cite{LW2012}. They gave an estimation in the similar settings with a fixed second fundamental form $2n$.

\begin{theorem}\label{thm:eigenComp}
Assume $M^n$ satisfies the condition in \autoref{thm:kahlercomp} with $\KH(h)\ge 0$ and also assume $M$ is compact. Denote the first Dirichlet eigenvalue of the Laplacian by $\lambda_1$. Then
$$
\lambda_1\ge \Big(\frac{\KH(h)}{2}\Big)^2.\\
$$
\end{theorem}

It is my pleasure to thank Werner Ballmann for useful comments.

\section{Proof of the \autoref{thm:k} and \autoref{thm:rigidity}}
In this note, all proofs are given in an uniform flavor, i.e. independent of the sign of $k$. Hence let's recall some definitions.
\begin{definition}\label{def:sincos}
Given a real constant $k$, we let $\sk_{k}$ denote the solution to the ordinary differential equation
\begin{equation}\label{eq:kJacobi}
\left\{
\begin{aligned}
\phi''+k\phi&=0,\\
\phi(0)=0,\quad &\phi'(0)=1
\end{aligned} \right.
\end{equation}
Setting $\ck_{k}(t)=\sk'_{k}(t)$, we clearly get $\ck'_{k}(t)=-k\sk_{k}(t)$ and $\ck_k$ satisfies $\ck_k''+k\ck_k=0$ , with initial condition $\ck_{k}(0)=1, \ck'_{k}(0)=0$.
\end{definition}

The explicit expressions are given by
\begin{equation} \label{eq:sk}
\sk_{k}(t)= \left\{ \begin{aligned}
&\frac{\sinh{(\sqrt{|k|}t)}}{\sqrt{|k|}} \quad&{\rm for }\quad k<0, \\
&t \quad \quad&{\rm for }\quad k=0,\\
&\frac{\sin{(\sqrt{k}t)}}{\sqrt{k}}\quad &{\rm for } \quad k>0,\\
\end{aligned} \right.
\end{equation}
\begin{equation} \label{eq:ck}
\ck_{k}(t)= \left\{ \begin{aligned}
&\cosh{(\sqrt{|k|}t)} \quad&{\rm for }\quad k<0 \\
&1 \quad \quad&{\rm for }\quad k=0\\
&\cos{(\sqrt{k}t)}\quad &{\rm for } \quad k>0\\
\end{aligned} \right.
\end{equation}
We sum up several basic formulas of $\sk_k$ and $\ck_k$, the proof is straightforward calculations.
\begin{prop}
\begin{equation}\label{eq:cksksum}
\begin{aligned}
\sk_k&'=\ck_k;\\
\ck_k'&=-k\sk_k;\\
\ck_k(A+B)&=\ck_k(A)\ck_k(B)-k\sk_k(A)\sk_k(B);\\
\sk_k(A+B)&=\sk_k(A)\ck_k(B)+\ck_k(A)\sk_k(B).
\end{aligned}
\end{equation}
\end{prop}

\begin{proof}[Proof of the \autoref{thm:k}]
For any fixed $\ell >0$ (if $k\ge 0$ we also assume $\ell<\pi/\sqrt{k}$), we consider the differential equation
\begin{equation} \label{eq:1}
\left\{
\begin{aligned}
         f''(s)+kf(s) &= 0 \\
                  f(0)&=0 \\
                  f(\ell)&=1\\
\end{aligned} \right.
\end{equation}
One easily verified
\begin{equation} \label{eq:uniJacobi}
f(s)=\frac{\sk_{k}(s)}{\sk_{k}(\ell)}
\end{equation}
is the solution.
For any point $x\in M$, let $p\in \partial M$ such that $\dis(p, x)=\rho(x)$. Let $\gamma : [0, \ell]\to M$ be a shortest geodesic joining $x$ to $p$, i.e. $\gamma(0)=x, \gamma(\ell)=p$ and $\ell =\rho(x)$. By the first variation formula, $\gamma'(\ell)\perp T_p(\partial M)$. For and vector field $V$ normal to $\gamma$, the second variational formula of arc-length gives
\begin{equation}\label{eq:2ndvar}
\begin{aligned}
L''(V, V) &=\int_0 ^\ell \Big( |V'(s)|^2 -|V(s)|^2 K(\gamma'(s), V(s))\Big )ds \\
&\quad -\langle \nabla_{V(\ell)} (\gamma'(\ell)), V(\ell) \rangle \\
&\ge 0,
\end{aligned}
\end{equation}
where $K(v, w)$ denotes the sectional curvature of the plane spanned by $v$ and $w$. Choose an orthonormal basis $E_1, \cdots, E_{n-1}$ of $T_p(\partial M)$. Let $E_i(s)$ be the parallel transportation of $E_i$ along $\gamma$. Define
$$
V_i(s)=f(s)E_i(s)
$$
Let $V=V_i$ in \eqref{eq:2ndvar}, sum over $i$ for $1$ to $(n-1)$ and apply the boundary condition of \eqref{eq:1}, one get
$$
\int_0^\ell \Big( (n-1) f'^2(s)-f^2(s)\Ric (\gamma'(s))\Big)ds -H_{p}\ge 0.
$$
where $\Ric(\gamma'(s))$ denotes the Ricci curvature along the direction $\gamma'(s)$ and $H_p$ denotes the mean curvature of $\partial M$ at $p$ w.r.t. the inner normal direction $-\gamma'(\ell)$:
$$
H_p=\sum_{i=1}^{n-1} \langle \nabla_{V_i(\ell)} \gamma'(\ell), V_i(\ell) \rangle.
$$
Integration by parts for the term $f'^2$ and make use of the differential equation \eqref{eq:1}, one get:
$$
(n-1)[ff']|_0^\ell-H_p\ge 0.
$$
Since $H\ge (n-1)h>0$, we have:
\begin{equation}\label{eq:ine}
f'(\ell)\ge h.
\end{equation}
That is $\ck_{k}(\ell)/\sk_{k}(\ell) \ge h$. Hence the theorem follows from the explicit expressions \eqref{eq:sk} and \eqref{eq:ck}.
\end{proof}

\begin{remark}
It can be seen from the proof that the condition $h>\sqrt{|k|}$ is essential to get an upper bound when solving \eqref{eq:ine}. For positive $k$, the estimates also holds for mean curvature bounded by a negative constant.
\end{remark}

In \cite{CG1971}, Cheeger and Gromoll proved the celebrate splitting theorem for open manifold with non-negative Ricci curvature, the key step is that the sum two Busemann functions constructed out of a line is super-harmonic and achieves a minimum along this line, hence the minimal principle of super-harmonic function implies it must be constant. Our proof of the Rigidity \autoref{thm:rigidity} is similar to Cheeger-Gromoll's proof. Note Li's proof of the rigidity for $k=0$ given in \cite{Li2012} is along a slightly different line, where a rigidity theorem for isoperemetric inequality plays an important role.
\begin{proof}[Proof of Rigidity \autoref{thm:rigidity}]
Since $\partial M$ is compact, \autoref{thm:k} implies that $M$ itself is compact. Hence there exists $x_0\in M$ such that $\rho(x_0)=h=\max_M (\rho(x))$. Define
$$
r(x)=\dist(x, x_0).
$$
The classical Laplacian comparison theorem implies that
\begin{equation}\label{eq:LapR}
\Delta r(x)\le \H(r(x)).
\end{equation}
\autoref{thm:realLapComp} implies
\begin{equation}\label{eq:LapRoh}
\Delta \rho(x)\le -\H(h-\rho(x)).
\end{equation}
Follow the idea of Cheeger-Gromoll mentioned above, one define:
$$
F(x)=r(x)+\rho(x)-h.
$$
By triangle inequality, we clearly have $F(x)\ge 0$ for $x\in M$. Let $\gamma: [0, h]\to M$ be a length minimizing geodesic joining $x_0$ and $p\in\partial M$. Hence $F(\gamma(t))\equiv 0$. One calculate
\begin{equation}\label{eq:LapF}
\begin{aligned}
\Delta F(x)&=\Delta r(x) +\Delta \rho(x)\\
&\le \H(r(x))-\H(h-\rho(x))\\
&\le \H(r(x))-\H(r(x))\\
&\le 0
\end{aligned}
\end{equation}
where we used \eqref{eq:LapR} and \eqref{eq:LapRoh} for the first inequality. The second inequality follows by the triangle inequality $h-\rho(x)\le r(x)$ and the fact that $\H(t)$ is monotone decreasing. Therefore $F$ is a nonnegative super-harmonic function on $M$ and achieves the minimal $0$ at some interior point $\gamma(s)$ for $s\in (0, h)$. Hence $F$ must be identically $0$. This shows the smoothness of the $r$ and $\rho$ in $M- \{x_0, \partial M\}$. In fact if $x\in M- \{x_0, \partial M\}$, then $x$ can be jointed to both $x_0$ and $\partial M$ by geodesic segments $\sigma_1$ and $\sigma_2$. If we put this two segment together then it has length $h=\dist(x_0, \partial M)$, such a segment must be smooth. Hence $r$ is smooth and $\rho=h-r$ is also smooth. Also note that the geodesic never bifurcate, which implies the uniqueness of such geodesic connecting $x_0$ to $\partial M$ passing through $x$.  Hence $M=D_M(x_0, h)$, $\partial M=\partial D_M(x_0, h)$ and moreover
$$
\exp: B_{T_{x_0}M}(o, h)\to M
$$
is a diffeomorphism.
The equality $\Delta F=0$ also implies
$$
\Delta r=H_k(r)=(n-1)\frac{\ck_k(r)}{\sk_k(r)},
$$
Taking derivative of $\H(r)$ with respect to the direction $\partial_r:=\partial/\partial r$ one get
\begin{equation}\label{eq:hessiancal}
\begin{aligned}
-(n-1)k&=\partial_r(\Delta r) +\frac{(\Delta r)^2}{n-1}\\
&\le \partial_r(\Delta r) + |\hess r|^2\\
&= -\Ric(\partial_r, \partial_r)\\
&\le -(n-1)k,
\end{aligned}
\end{equation}
where the first inequality is Cauchy-Schwartz and the second equality is Weitzenb\"ock formula applied for distance function $r$. Hence all inequalities in \eqref{eq:hessiancal} are equalities. In particular equality in Cauchy-Schwartz implies that Hessian is diagonal: $\hess r = \ck_k(r)/\sk_k(r) g_r$ for $0<r<h$. Therefore the metric can be written in polar coordinate as
$$
g=dr^2+\sk_k^2(r) ds_{n-1}^2
$$
where $ds_{n-1}^2$ is the standard metric on $(n-1)$-dimensional sphere of sectional curvature 1. Hence $M$ is isometric to the disk of radius $h$ in space form $\SS^n_k$.
\end{proof}

\section{Comparison Theorems for \kahler\ Manifolds}\label{sec:kahler}
We first give a proof similar to the proof of \autoref{thm:k}.
\begin{proof}[Proof of \autoref{thm:kahler}]
Fix $x\in M$. Let $p\in \partial M$ be a point such that $\dist(x, p)=\rho(x)$. It is know that under the assumption $\bisec \ge 2$ the diameter of $M$ is less than $\pi/2$ see \cite{LW2005}. Hence we can assume $\ell \le \pi/2$ for the case $k=1$. Let $\gamma: [0, \ell] \to M$ be a geodesic such that $\gamma(0)=x, \gamma(\ell)=p$ and $\ell=\rho(x)$. Choose an orthonormal basis of $T_pM$:
$$
\{E_1, E_2, \cdots, E_{2n-1}, E_{2n}\},
$$
satisfying $E_{2k}=JE_{2k-1}$ for $k=1, \cdots, (n-1)$ and $E_{2n}=\gamma'(\ell)$. Denote by $E_i(s)$ the parallel translation of $E_i$ along $\gamma$. Define:
$$
f(s)=\frac{\sk_{k}(s)}{\sk_{k}(\ell)}, \quad g(s)=\frac{\sk_{4k}(s)}{\sk_{4k}(\ell)}.
$$
and
$$
V_i(s)=f(s)E_i(s),\ \text{for } i=1, \cdots, 2n-2, \quad V_{2n-1}(s)=g(s)E_{2n-1}(s).
$$
By the second variational formula one have
\begin{equation} \label{eq:3}
\begin{split}
0&\le \sum_{i=1}^{2n-1} \delta^2(V_i, V_i) \\
&=\int_0^\ell \Big ( (2n-2) f'^2(s) -\sum_{i=1}^{2n-2} f^2(s) K(\gamma', E_i) \\
&\quad\quad+ g'^2(s) - g^2(s) K(\gamma', J\gamma') \Big )ds- H_p\\
& = (2n-2)\ck_k(\ell)/\sk_k(\ell) +\ck_{4k}(\ell)/\sk_{4k}(\ell) - H_p\\
\end{split}
\end{equation}
i.e.
$$
\KH(\ell) \ge \KH(h).
$$
By the monotonicity of $\KH$, we have $\ell \le h$.
\end{proof}

The distance estimate can also be derived from the following Hessian comparison theorem. But let us recall first some basic definitions, which can be found in \cite{BC1964}. Let $N\in M$ a submanifold, then the $N$-Jacobi field along a geodesic $\gamma: [0, b]\to M$ with $\gamma(0)=p\in N$ and $\gamma'(0)\perp T_p(N)$ is the unique Jacobi field $J$ satisfying:
$$
J(0)\in T_p(N), \nabla_{\gamma'(0)}J (0) - S_{\xi} J(0) \in T_p^{\perp}N,
$$
where $S_{\xi}$ denotes the shape operator of $N$ with respect to the normal vector $\xi$, i.e. $\langle S_{\xi}V, W\rangle=\langle \nabla_V\xi, W\rangle$ and $T_p^{\perp}N$ the normal bundle of $N$ in $M$. Recall also that the index form $I_{\gamma}(,)$ associated with $\gamma$, for any vector field $V\perp \gamma'$ along $\gamma$, it is defined by:
$$
I_{\gamma}(V, V)=\langle S_{\gamma'(0)}V(0), V(0) \rangle +\int_0^\ell \Big |V'|^2+|V|^2K(\gamma', V) \Big) ds.
$$
Note that the second fundamental form with respect to the inner normal direction of $\partial M$ is defined by  $\II(V(0), V(0))=-\langle S_{\gamma'(0)}V(0), V(0) \rangle$. We have the following
\begin{prop}[The basic inequality. cf. \cite{BC1964}]\label{prop:Minimality}
Suppose there is no focal point of $N$ on $\gamma (0, \ell]$. For all vector field $V$ along $\gamma$ with $V(0)\in T_pN$, there is an unique $N$-Jacobi field $J$ such that
$J(\ell) = V(\ell)$. Moreover $I_{\gamma}(V)\ge I_{\gamma}(J)$ and equality occurs if and only if $V=J$.
\end{prop}

\autoref{thm:kahlercomp} is a consequence of the following complex Hessian comparison theorem.
\begin{theorem}\label{thm:kahlerHessian}
Let $M^n$ be a \kahler\ manifold of complex dimension $n$ with holomorphic bisectional curvature $\bisec \ge 2k$, where $k=1, 0$ or $-1$. If the second fundamental form of $\partial M$ with respect to inner normal direction $\nu$ satisfy $\II(J\nu, J\nu)\ge \ck_{4k}(h)/\sk_{4k}(h)$ and $\II(V, V) + \II(JV, JV) \ge 2 \ck_{k}(h)/\sk_{k}(h)$ for any $V\perp J\nu$, if $k=1$ we assume $0<h<\pi/2$. Let $\rho(x)=\dis(x, \partial M):=\ell$ be the distance to the boundary. Then for all $\ell<h$, we have
\begin{equation}\label{eq:HessCompKah}
\begin{aligned}
\hess\rho(J\nu, J\nu)&\le -\frac{\ck_{4k}(h-\ell)}{\sk_{4k}(h-\ell)}\\
\hess\rho(V, V) +\hess\rho(JV, JV) &\le -2\frac{\ck_k(h-\ell)}{\sk_k(h-\ell)}
\end{aligned}
\end{equation}
In particular,
\begin{equation}\label{eq:LapCompKah}
\Delta(\rho)(x)\le - \KH(h-\rho(x)) \le -\KH(h)
\end{equation}
\end{theorem}
\begin{proof}
For any $x\in M$ let $p\in \partial M$ be a point such that $\rho(x)=\dist(x, \partial M)=:\ell$. Let $\gamma: [0, \ell]\to M$ be a geodesic from $p$ to $x$. By \autoref{thm:kahler}, $\ell \le h$. It suffices to estimate the Laplacian for non cut point $x$. Let $f$ be the solution of the differential equation:
\begin{equation}\label{eq:comple4K}
\left\{
\begin{aligned}
f'' + 4kf&=0,\\
f(0)=1, \quad &f'(0)=-h_1,
\end{aligned} \right.
\end{equation}
where $h_1=\ck_{4k}(h)/\sk_{4k}(h)$. The explicit solution is
$$
f(s)= \ck_{4k}(s)-h_1 \sk_{4k}(s).
$$
Also let $g$ be the solution of
\begin{equation}\label{eq:compleK}
\left\{
\begin{aligned}
g'' + kg&=0,\\
g(0)=1, \quad &g'(0)=-h_2,
\end{aligned} \right.
\end{equation}
where $h_2=\ck_{k}(h)/\sk_{k}(h)$. The explicit solution is
$$
g(s)= \ck_{k}(s)-h_2 \sk_{k}(s).
$$

Since $\ell<h$, one easily verify that $f(\ell)$ and $g(\ell)$ are non-zero, hence we can define
$$
\tilde{f}(s)=f(s)/f(\ell), \quad \tilde{g}(s)=g(s)/g(\ell).
$$
and let $\{E_1, \cdots, E_{2n}\}$ be an orthonormal basis of $T_pM$ such that $JE_{2i-1}=E_{2i}$ for $i=1, \cdots, (n-1)$ and $E_{2n}=\gamma' (0)$. Parallel translate the $E_i$ along $\gamma$, one get $E_i(s)$ and define
$$
V_i(s)=\tilde{g}(s) E_i, \ \text{for }\ i =1, \cdots, 2n-1, \ \text{and }\ V_{2n-1}(s)=\tilde{f}(s)E_{2n-1}.
$$

Let $J_i$ be the unique $\partial M$-Jacobi field with $J_i(\ell)=E_i(\ell)$. Hence
$$
\nabla^2\rho (E_i(\ell), E_i(\ell))=\nabla^2\rho (J_i(\ell), J_i(\ell))=\langle \nabla_{\gamma'(\ell)}J_i, J_i \rangle = I_{\gamma}(J_i, J_i),
$$

By \autoref{prop:Minimality}, one have $\nabla^2\rho (E_i(\ell), E_i(\ell))\le I(V_i, V_i)$. Hence it suffices to estimate $I(V_i, V_i)$ from above. In fact we have
\begin{equation} \label{eq:f}
\begin{split}
I_{\gamma}(V_{2n-1}, V_{2n-1})&= -\tilde{f}^2(0)\II (E_{2n-1}, E_{2n-1})\\
&\quad+ \int_0^\ell \Big( (\tilde{f}')^2-(\tilde{f})^2K(\gamma', E_{2n-1})\Big) ds\\
&\le - \tilde{f}^2(0) h_1 + \int_0^\ell \Big( (\tilde{f}')^2-4k(\tilde{f})^2\Big) ds\\
&= - \tilde{f}^2(0) h_1 + \tilde{f}'\tilde{f}|_0^{\ell}\\
&= f'(\ell)/f(\ell)\\
&=\frac{-4k\sk_{4k}(\ell)-\ck_{4k}(\ell)\ck_{4k}(h)/\sk_{4k}(h)}{\ck_{4k}(\ell)-\sk_{4k}(\ell)\ck_{4k}(h)/\sk_{4k}(h)}\\
&=-\frac{\ck_{4k}(h-\ell)}{\sk_{4k}(h-\ell)}
\end{split}
\end{equation}
Similar calculation shows that
\begin{equation} \label{eq:g}
I_{\gamma}(V_{2i-1}, V_{2i-1}) + I_{\gamma}(V_{2i}, V_{2i}) \le 2 g'(\ell)/g(\ell)=-2\frac{\ck_k(h-\ell)}{\sk_k(h-\ell)}
\end{equation}

Sum up for $i=1$ to $2n-1$, one get
\begin{equation} \label{eq:comp1}
\begin{split}
\Delta\rho &\le I_{\gamma}(V_{2n-1}, V_{2n-1})+ \sum_{i=1}^{n-1} I_{\gamma}(V_{2i-1}, V_{2i-1}) + I_{\gamma}(V_{2i}, V_{2i})  \\
&\le f'(\ell)/f(\ell) + 2(n-1) g'(\ell)/g(\ell)\\
&\le -\KH(h-\ell).
\end{split}
\end{equation}
The second inequality in \eqref{eq:LapCompKah} follows from the \autoref{thm:kahler} and monotonicity of $\KH$.
\end{proof}

Using the same idea of \cite{LW2012}, we can estimate $\lambda_0$ as follows.

\begin{proof}[Proof of \autoref{thm:eigenComp}]
Let $f$ be the eigenfunction correspond to $\lambda_0$. i.e. $f$ satisfies
\begin{equation}\label{eq:Laplace}
\left\{
\begin{aligned}
\Delta f+\lambda_0 f&=0,\quad &\text{on } M,\\
f&=0,\quad &\text{on } \partial M.
\end{aligned} \right.
\end{equation}
We can assume $f>0$ on $M$. For any $C^2$ function $g: M\to \RR$, we let
$$
F(x)=f(x)e^{g(x)}.
$$
Clearly $F\ge 0$ on $M$ and $F|_{\partial M}=0$. Hence by compactness of $M$,  $F$ reaches its maximum at an interior point, say $p_0\in M$. Hence at $p_0$ we have
\begin{equation*}
\nabla F(p_0)=(\nabla f \cdot e^g + fe^g \nabla g)|_{p_0}=0;
\end{equation*}
where $\nabla f$ denotes the gradient vector of $f$. Hence at $p_0$:
\begin{equation*}
\nabla f=-f\nabla g.
\end{equation*}
One calculates at $p_0$:
\begin{equation}
\begin{aligned}
0\ge \Delta F&=\Delta(fe^g)\\
&= e^g(\Delta f +2 \langle \nabla f, \nabla g\rangle +f \Delta g + f\langle \nabla g, \nabla g\rangle )\\
&= e^g(-\lambda_0 f +2\langle -f\nabla g, \nabla g\rangle +f \Delta g + f\langle \nabla g, \nabla g\rangle )\\
&= e^g f(-\lambda_0 -|\nabla g|^2 +\Delta g).
\end{aligned}
\end{equation}
Therefore
\begin{equation}\label{eq:lapcomp01}
\lambda_0\ge (\Delta g-|\nabla g|^2)|_{p_0}.
\end{equation}
for and $g\in C^2(M)$ and $p_0$ depends on $g$. Hence to get a lower bound for $\lambda_0$ one let $g(x)=-c\rho(x)$, where $\rho(x)=\dist(x, \partial M)$ and $c$ is some positive constant to be determined. Hence by \autoref{thm:kahlercomp}, the following inequalities hold in barrier sense
$$
\Delta g \ge c \KH(h), \quad |\nabla g|\le c
$$
Hence by \eqref{eq:lapcomp01}, the following inequality holds for all $c\ge 0$
\begin{equation}\label{eq:lapcomp02}
\lambda_0\ge c(\KH(h) -c).
\end{equation}
One easily see the maximum of the right hand side of \eqref{eq:lapcomp02} is
$$
\Big ( \frac{\KH(h)}{2}\Big)^2,
$$
which is achieved when $c=\KH(h)/2\ge 0$.
\end{proof}

\bibliographystyle{amsalpha}

\bigskip
\noindent
{\small Jian Ge}\\
{\small Max Planck Institute for Mathematics } \\
{\small Vivatsgasse 7}\\
{\small 53111 Bonn, Germany} \\
{\small jge@mpim-bonn.mpg.de} \\
\end{document}